\documentclass[11pt,leqno]{article}

\usepackage[active]{srcltx}

\usepackage{amsfonts,latexsym,amsmath,amssymb,amsthm}
\usepackage{fullpage}

\usepackage{bbm} 
\newtheorem{theorem}{Theorem}[section]
\newtheorem{lemma}[theorem]{Lemma}

\newtheorem{remark}[theorem]{Remark}

\theoremstyle{definition}

\theoremstyle{remark}

\newtheorem*{note*}{Note}

\numberwithin{equation}{section}

\makeatletter
\newcommand{\rank}{\mathop{\operator@font rank}}
\newcommand{\conv}{\mathop{\operator@font conv}}
\newcommand{\vol}{\mathrm{vol}}
\newcommand{\onetagright}{\tagsleft@false}
\makeatother

\makeatletter 
\def\fullwidthdisplay{\displayindent\z@ \displaywidth\columnwidth}
\edef\@tempa{\noexpand\fullwidthdisplay\the\everydisplay}
\everydisplay\expandafter{\@tempa}
\makeatother
\makeatletter

\newcommand{\B}{\mathrm{B}}

\makeatother

\newcommand{\ls}{\leqslant}
\newcommand{\gr}{\geqslant}
\renewcommand{\epsilon}{\varepsilon}

\usepackage{color}

\usepackage{hyperref}

\begin{document}
\small

\title{\bf Affine quermassintegrals of random polytopes}

\author{Giorgos Chasapis and Nikos Skarmogiannis}

\date{}

\maketitle

\begin{abstract}
\footnotesize A question related to some conjectures of Lutwak about the affine quermassintegrals of a convex body $K$ in ${\mathbb R}^n$
asks whether for every convex body $K$ in ${\mathbb R}^n$ and all $1\ls k\ls n$
$$\Phi_{[k]}(K):=\vol_n(K)^{-\frac{1}{n}}\left (\int_{G_{n,k}}\vol_k(P_F(K))^{-n}\,d\nu_{n,k}(F)\right )^{-\frac{1}{kn}}\ls c\sqrt{n/k},$$
where $c>0$ is an absolute constant. We provide an affirmative
answer for some broad classes of random polytopes. We also discuss upper bounds for $\Phi_{[k]}(K)$ when $K=B_1^n$, the unit ball of $\ell_1^n$, and explain how this special instance has implications for the case of a general unconditional convex body $K$.
\end{abstract}

\section{Introduction}

The \textit{affine quermassintegrals} of a convex body $K$ in ${\mathbb R}^n$ were introduced by Lutwak in \cite{Lutwak-1984}:
they are defined by
\begin{equation*}\Phi_{n-k}(K)=\frac{\omega_n}{\omega_k}\left
(\int_{G_{n,k}}\vol_k(P_F(K))^{-n}d\nu_{n,k}(F)\right )^{-1/n}\end{equation*}
for $1\ls k\ls n-1$, where $\nu_{n,k}$ is the Haar probability measure on the Grassmannian $G_{n,k}$ of all $k$-dimensional
subspaces of ${\mathbb R}^n$ and $\omega_k$ is the volume of the Euclidean unit ball $B_2^k$ in ${\mathbb R}^k$. In what follows, we will also adopt the notational convention $\Phi_0(K) = \vol_n(K)$ and $\Phi_n(K) = \omega_n$.
Grinberg proved in \cite{Grinberg-1990} that these quantities are invariant under volume preserving affine transformations.
Lutwak conjectured in \cite{Lutwak-1988} that the affine quermassintegrals satisfy the inequalities
\begin{equation}\label{eq:lutwak-1}\omega_n^j\Phi_{n-j}(K)^{k}\ls \omega_n^k\Phi_{n-k}(K)^{j}\end{equation}
for all $0\ls k\ls j\ls n$, with equality when $k<j$ if and only if $K$ is an ellipsoid, and, in particular for $j=n$, that
\begin{equation}\label{eq:lutwak-2}\omega_n^{\frac{n-k}{n}}\vol_n(K)^{\frac{k}{n}}\ls\Phi_{n-k}(K)\end{equation}
for all $0\ls k\ls n$ with equality if and only if $K$ is an ellipsoid (see \cite[Chapter 9]{Gardner-book} for related conjectures about dual affine quermassintegrals and references).

The following variant of the quantity $\Phi_{n-k}$ was considered by Dafnis and Paouris in \cite{Dafnis-Paouris-2012}: We define, for every convex body $K$ in $\mathbb{R}^n$ and every $1\ls k\ls n$, the \textit{normalized $k$-th affine quermassintegral} of $K$ by
\[
\Phi_{[k]}(K) := \vol_n(K)^{-\frac{1}{n}}\left (\int_{G_{n,k}}\vol_k(P_F(K))^{-n}\,d\nu_{n,k}(F)\right )^{-\frac{1}{kn}}.
\]
Note that $\Phi_{[k]}(K) = \vol_n(K)^{-\frac{1}{n}}\left(\frac{\omega_k}{\omega_n}\right)^{\frac{1}{k}}\Phi_{n-k}(K)^{\frac{1}{k}}$, so the conjectured inequality \eqref{eq:lutwak-2} can be equivalently restated as
\begin{equation}\label{eq:lutwak-3}
\Phi_{[k]}(K) \gr \Phi_{[k]}(B_2^n).
\end{equation}
When $k=1$ the above inequality follows by the Blaschke-Santal\'{o} inequality, which states that the volume product of a convex body $K$ with
center of mass at the origin and its polar $K^\circ$ is maximal if $K$ is an ellipsoid:
\begin{equation}\label{eq:blaschke-santalo}
\vol_n(K)\cdot\vol_n(K^\circ) \ls \omega_n^2.
\end{equation}
In the case $k=n-1$, note that
\[
\Phi_{[n-1]}(K) = \vol_n(K)^{-\frac{1}{n}}\left(\frac{\vol_n(\Pi^\ast K)}{\omega_n}\right)^{-\frac{1}{n(n-1)}},
\]
where $\Pi^\ast K$ is the polar projection body of $K$ (this is the polar of the convex body $\Pi K$, defined by $h_{\Pi K}(\theta) = \vol_{n-1}(P_{\theta^\perp}K)$ for every $\theta\in S^{n-1}$). Then \eqref{eq:lutwak-3} follows by the Petty projection inequality \cite{Petty}:
\begin{equation}\label{eq:petty-proj}
\vol_n(K)^{n-1}\vol_n(\Pi^\ast K) \ls \left(\frac{\omega_n}{\omega_{n-1}}\right)^n.
\end{equation}

The authors in \cite{Dafnis-Paouris-2012} studied an isomorphic variant of Lutwak's conjecture; they ask if there exist
absolute constants $c_1,c_2>0$ such that for every convex body $K$ in $\mathbb R^n$ and any $1\ls k \ls n-1$,
\begin{equation}\label{eq:phi-conj}c_1\sqrt{n/k}\ls  \Phi_{[k]}(K) \ls c_2\sqrt{n/k}\end{equation}
(recall that $\omega_k^{1/k}$ is of the order of $k^{-1/2}$).
Note that in the case $k=1$, \eqref{eq:phi-conj} follows by the Blaschke-Santal\'{o} and the reverse Santal\'{o} inequality of Bourgain and Milman \cite{BM}, while in the case $k=n-1$ the conjectured rate of growth for $\Phi_{[n-1]}(K)$ is again true, by the Petty projection inequality and its reverse, proved by Zhang \cite{Zhang}.

The left hand side of \eqref{eq:phi-conj} was proved by Paouris and Pivovarov in \cite{Paouris-Pivovarov-2013}; it confirms \eqref{eq:lutwak-1} in an isomorphic sense.

\begin{theorem}[Paouris-Pivovarov]\label{th:main-3}Let $K$ be a convex body in ${\mathbb R}^n$ and $1\ls k \ls n$. Then,
\begin{equation}\label{eq:main-17}\Phi_{[k]}(K)\gr c\sqrt{n/k}.\end{equation}
\end{theorem}
The proof of Theorem \ref{th:main-3} relies on a duality argument, that employs the Blaschke Santal\'{o} inequality \eqref{eq:blaschke-santalo} as well as its reverse, combined with an isoperimetric-type inequality on moments of sections of a convex body proved by Grinberg \cite{Grinberg-1990}, according to which
\begin{equation}\label{eq:grinberg}
\vol_n(K)^{-\frac{1}{n}} \left(\int_{G_{n,k}} \vol_{k}(K\cap F)^{n} \,d\nu_{n,k}(F)\right)^{\frac{1}{kn}} \ls \frac{\omega_k^{1/k}}{\omega_n^{1/n}}.
\end{equation}

The main question that we discuss in this note is related to the upper bound in \eqref{eq:phi-conj}. An almost optimal estimate (up to a $\log n$-term) was given by Dafnis and Paouris in \cite{Dafnis-Paouris-2012}. Let us briefly recall their argument: The Aleksandrov inequalities (see \cite[Sections 20.1-20.2]{Burago-Zalgaller-book}
and \cite[Section 6.4]{Schneider-book}) imply that if $K$ is a convex body in ${\mathbb R}^n$ then the sequence
\begin{equation}\label{eq:main-1}Q_k(K)=\left
(\frac{1}{\omega_k}\int_{G_{n,k}}\vol_k(P_F(K))\,d\nu_{n,k}(F)\right
)^{1/k}\end{equation}is decreasing in $k$. In particular, for any $1\ls k\ls n-1$ we have $Q_k(K)\ls Q_1(K)$, which may be written
in the equivalent form
\begin{equation}\label{eq:main-2}\left (\frac{1}{\omega_k}\int_{G_{n,k}}\vol_k(P_F(K))\,d\nu_{n,k}(F)\right )^{\frac{1}{k}}\ls w(K),\end{equation}
where $w(K)$ is the mean width of $K$. Then, by H\"{o}lder's inequality,
\begin{equation*}\left (\int_{G_{n,k}}\vol_k(P_F(K))^{-n}\,d\nu_{n,k}(F)\right )^{-\frac{1}{kn}}
\ls \left (\int_{G_{n,k}}\vol_k(P_F(K))\,d\nu_{n,k}(F)\right )^{\frac{1}{k}}\ls \omega_k^{1/k}w(K).\end{equation*}
Since the term on the left hand side of this inequality is invariant under volume preserving affine transformations, we may assume
that $K$ has minimal mean width, and it is known that in this case we have $w(K)\ls c\sqrt{n}\log n\,\vol_n(K)^{1/n}$ for some absolute
constant $c>0$ (see \cite[Chapter~6]{AGA-book}). Combining the above with the fact that $\omega_k^{1/k}$ is of the order of $1/\sqrt{k}$, we get
\begin{equation}\label{eq:DP}\Phi_{[k]}(K)\ls c_2\sqrt{n/k}\log n.\end{equation}
It was also shown in \cite{Dafnis-Paouris-2012} that
\begin{equation*}\Phi_{[k]}(K)\ls c_3(n/k)^{3/2}\sqrt{\log{(en/k)}}.\end{equation*}
In other words, if $k$ is proportional to $n$ then the upper bound for $\Phi_{[k]}(K)$ is of the order of $1$. The main question that remains open is whether the $\log n$-term in \eqref{eq:DP} can actually be dropped.

In this note we study this question for some broad classes of random polytopes. First, we provide an affirmative answer to the problem for the
class of symmetric random polytopes with at most $e^{\sqrt{n}}$ vertices uniformly distributed on a convex body.
By the affine invariance of the problem, we may concentrate on the isotropic case. Let $N\gr n$ and $x_1,\ldots,x_N$ be independent
random vectors chosen uniformly from an
isotropic convex body $K$ in $\mathbb{R}^n$ (that is, with respect to the normalized Lebesgue measure on $K$). Consider the symmetric random polytope
$$K_N:=\conv\{\pm x_1,\ldots, \pm x_N\}.$$

\begin{theorem}\label{th:random}
Let $K$ be an isotropic convex body in $\mathbb{R}^n$, $1\ls k\ls n$ and $n^2\ls N\ls e^{\sqrt{n}}$.
If $x_1,\ldots,x_N$ are independent random vectors chosen uniformly from $K$, then
$$\Phi_{[k]}(K_N) \ls c\sqrt{n/k}$$ for some absolute constant $c>0$,
with probability greater than $1-\frac{2}{N}$.
\end{theorem}

Next, we consider the case of the cone probability measure $\mu_K$ on the boundary $\partial (K)$ of a convex body $K$, which is
defined by
$$\mu_K(B)=\frac{\vol_n(\{rx:x\in B,0\ls r\ls 1\})}{\vol_n(K)}$$
for all Borel subsets $B$ of $\partial (K)$. For any $N\gr n$ we consider independent random points $x_1,\ldots ,x_N$ distributed according to $\mu_K$
and the random polytope $M_N={\rm conv}\{\pm x_1,\ldots ,\pm x_N\}$. We provide a description of the ``asymptotic shape" of $M_N$ which is
parallel to the available description for $K_N$; this can be done with suitable modifications of the theory developed in \cite{DGT} and \cite{DGT2}.
This allows us to prove the analogue of Theorem~\ref{th:random} for this model too.

\begin{theorem}\label{th:random-cone}
Let $K$ be an isotropic convex body in $\mathbb{R}^n$, $1\ls k\ls n$ and $n^2\ls N\ls e^{\sqrt{n}}$.
If $x_1,\ldots,x_N$ are independent random vectors with distribution $\mu_K$, then
$$\Phi_{[k]}(M_N) \ls c\sqrt{n/k}$$ for some absolute constant $c>0$,
with probability greater than $1-\frac{1}{N^2}$.
\end{theorem}

We also study a different model of random polytopes. Given $\beta >-1$, let $\nu_\beta$ be the probability measure supported on $B_2^n$, with density
$p_{n,\beta}(x) = c_{n,\beta}(1-\|x\|_2^2)^\beta $,
where $c_{n,\beta} := \pi^{-n/2}\frac{\Gamma\left(\beta+\frac{n}{2}+1 \right)}{\Gamma(\beta+1)}$.
Fix $N>n$, and let $x_1,\ldots,x_N$ be random vectors, chosen independently according to the measure $\nu_\beta$. The beta polytope
in ${\mathbb R}^n$ (with parameter $\beta $) is the random polytope
$$P_{N,n}^{\beta} := \conv\{x_1,\ldots,x_N\}.$$

\begin{theorem}\label{th:random-beta}
 Let $\beta>-1$ and $x_1,\ldots,x_N$ be independent random points in $\mathbb{R}^n$, distributed according to $\nu_\beta$. If
 $k\gr \log\left(n\left(1+\log\left(4\sqrt{\beta+\frac{n}{2}+1}\right)\right)\right)$ and $N\gr c_0^{\beta+\frac{n+1}{2}}$, where $c_0>0$ is an absolute constant, then
\[
\Phi_{[k]}(P_{N,n}^{\beta}) \ls c\sqrt{n/k}
\]
with probability greater than $1-e^{-k}$, where $c>0$ is an absolute constant.
\end{theorem}

In the last part of this note we study the quantities $\Phi_{[k]}(K)$ for the class of unconditional convex bodies $K$. The emphasis is drawn on the case $K=B_1^n$, since by known results of Bobkov and Nazarov (see Section \ref{sec:unco}) one can show that if $K$ is an unconditional convex body in $\mathbb{R}^n$, then, for every $1\ls k \ls n$,
\[
\Phi_{[k]}(K) \ls c\Phi_{[k]}(B_1^n)
\]
where $c>0$ is an absolute constant. Therefore, we only need to prove the following result for the case $K=B_1^n$.

\begin{theorem}\label{th:unco-1}Let $K$ be an unconditional convex body in ${\mathbb R}^n$. Then, for any $\log n\ls k\ls n$,
\begin{equation}\label{eq:unco-1}\Phi_{[k]}(K)\ls c\sqrt{n/k}\cdot\sqrt{1+\log (n/k)},\end{equation}
where $c>0$ is an absolute constant.
\end{theorem}

More generally, for any $p\neq 0$ one may consider the quantity
\begin{equation*}W_{[k,p]}(K)=\vol_n(K)^{-\frac{1}{n}}\left (\int_{G_{n,k}}\vol_k(P_F(K))^{p}\,d\nu_{n,k}(F)\right )^{\frac{1}{kp}}\end{equation*}
and study its behavior with respect to $p$, $n$ and $k$ in the case where $K$ is a convex body in ${\mathbb R}^n$ (note that $\Phi_{[k]}(K)=W_{[k,-n]}$).
In the unconditional case, studying the case $K=B_1^n$ and using the fact that $W_{[k,-p]}(K)\ls cW_{[k,-p]}(B_1^n)$
for all $p$, we provide bounds for the ``minimal value" of $p$ for which $W_{[k,-p]}(K)\ls c\sqrt{n/k}$.

\begin{theorem}\label{th:unco-2}Let $K$ be an unconditional convex body in ${\mathbb R}^n$. Then, for any $1\ls k\ls n$
and any $p\gr c_1(n-k)\log n$ we have
\begin{equation}\label{eq:unco-2}W_{[k,-p]}(K)\ls c_2\sqrt{n/k},\end{equation}
where $c_1,c_2>0$ are absolute constants.
\end{theorem}

\section{Notation and background on isotropic convex bodies}

We work in ${\mathbb R}^n$, which is equipped with a Euclidean structure $\langle\cdot ,\cdot\rangle $. We denote by $B_2^n$ and $S^{n-1}$ the Euclidean unit ball and sphere
in ${\mathbb R}^n$ respectively. We write $\sigma $ for the normalized rotationally invariant probability measure on $S^{n-1}$ and $\nu $
for the Haar probability measure on the orthogonal group $O(n)$. Let $G_{n,k}$ denote the Grassmannian of all $k$-dimensional
subspaces of ${\mathbb R}^n$. Then, $O(n)$ equips $G_{n,k}$ with a Haar probability measure $\nu_{n,k}$. We write $\vol_k$ for $k$-dimensional volume and $\|x\|_2$ for the Euclidean norm of $x$. The letters $c,c^{\prime }, c_1, c_2$ etc. denote absolute positive constants which may change from line to line. Since usually the exact numerical values of such absolute constants are not relevant, we further relax our notation:  $a\lesssim b$ will then mean ``$a\ls cb$ for some (suitable) absolute constant $c>0$'', and $a \asymp b$ will stand for ``$a \lesssim b \land a \gtrsim b$".

We refer to the book of Schneider \cite{Schneider-book} for basic facts from the Brunn-Minkowski theory and to the book of Artstein-Avidan, Giannopoulos and V. Milman \cite{AGA-book} for basic facts from asymptotic convex geometry.

\smallskip

A convex body in ${\mathbb R}^n$ is a compact convex subset $K$ of ${\mathbb R}^n$ with non-empty interior. We say that $K$ is
symmetric if $x\in K$ implies that $-x\in K$, and that $K$ is centered if its barycenter
is at the origin. The polar body of $K$ is denoted by $K^{\circ }$. The volume radius of $K$ is the quantity
${\rm vrad}(K)=(\vol_n(K)/\vol_n(B_2^n))^{1/n}$. Every convex body can be naturally associated to a probability measure $\lambda_K$ on $\mathbb{R}^n$, given by the normalized Lebesgue measure
\[
\lambda_K(A) := \frac{\vol_n(A\cap K)}{\vol_n(K)},
\]
for every measurable subset of $\mathbb{R}^n$. We call $\lambda_K$ the uniform probability measure on $K$.

The support function $h_K : \mathbb{R}^n\to \mathbb{R}$ of $K$ is defined by
$h_K(\xi ) = \max_{x\in K} \langle x,\xi \rangle $. The circumradius $R(K)$ is the radius of the smallest Euclidean ball enclosing $K$, that is
$R(K) := \min \{r>0 : K\subseteq rB_2^n\}$. Equivalently, $R(K) = \max_{\xi\in S^{n-1}} h_K(\xi)$.
The inradius $r(K)$ of $K$ is the radius of the largest Euclidean ball that lies inside $K$, i.e.
$r(K) := \max\{r>0 : rB_2^n \subseteq K\}$. As with $R(K)$, one can check that $r(K) = \min_{\xi \in S^{n-1}} h_K(\xi)$.
The mean width of $K$ is the average
$$w(K) := \int_{S^{n-1}} h_K(\xi)\,d\sigma(\xi).$$
More generally one can define, for any $q\in[-n,n]$, $q\neq0$,
\[
w_q(K) := \left(\int_{S^{n-1}} h_K(\xi)^q\,d\sigma(\xi) \right)^{1/q}.
\]
These quantities are usually referred to as the mixed widths of $K$.

A convex body $K$ in ${\mathbb R}^n$ is called isotropic if it has volume $1$, it is centered,
and its inertia matrix is a multiple of the identity matrix:
there exists a constant $L_K >0$ such that
\begin{equation*}\label{isotropic-condition} \int_K\langle x,\xi\rangle^2dx =L_K^2\end{equation*}
for every $\xi $ in the Euclidean unit sphere $S^{n-1}$. The hyperplane conjecture
asks if there exists an absolute constant $C>0$ such that
\begin{equation*}\label{HypCon} L_n:= \max\{ L_K:K\ \hbox{is isotropic in}\ {\mathbb R}^n\}\ls C\end{equation*}
for all $n\gr 1$. Bourgain proved in \cite{Bourgain-1991} that $L_n\ls
c\sqrt[4]{n}\log\! n$, while Klartag \cite{Klartag-2006}
obtained the bound $L_n\ls c\sqrt[4]{n}$. In the sequel we will need a number of notions introduced (and results proved by a series of authors) in works closely related to the above problem. We refer the reader to the book of Brazitikos, Giannopoulos, Valettas and Vritsiou \cite{BGVV-book} for an updated exposition of the theory of isotropic convex bodies (and log-concave measures) and more information on the hyperplane conjecture.

The $L_q$-centroid bodies were introduced, under a different normalization, by Lutwak and Zhang in \cite{LZ}, and studied by Lutwak, Yang and Zhang in \cite{LYZ}. Paouris was the first to exploit their properties from an asymptotic point of view. We shall use his notation and normalization:
If $K$ is a convex body in $\mathbb{R}^n$ with $\vol_n(K)=1$, for any $q\gr 1$ we define the $L_q$-centroid body of $K$, denoted $Z_q(K)$, via its support function
\[
h_{Z_q(K)}(\xi) := \|\langle \cdot,\xi\rangle\|_{L_q(K)} = \left( \int_K |\langle x,\xi\rangle|^q\,dx\right)^{1/q}, \hspace{10pt} \xi \in S^{n-1}.
\]
For $q=+\infty$, we define $Z_\infty(K) := \conv\{K,-K\}$. Some basic properties of this family of bodies are listed below:
\begin{itemize}
\item[\rm(a)] If $K$ is isotropic, then $Z_2(K)= L_KB_2^n$.
\item[\rm(b)] For all $1\ls p < q \ls \infty$ and $\xi\in S^{n-1}$ we have $\|\langle\cdot ,\xi\rangle\|_{L^q(\lambda_K)}
\ls c_1\frac{q}{p}\|\langle\cdot ,\xi\rangle\|_{L^p(\lambda_K)}$, and hence $Z_p(K) \subseteq Z_q(K) \subseteq c_1\frac{q}{p} Z_p(K)$,
where $c_1>0$ is an absolute constant.
\item[\rm(c)] If $K$ is centered, then $Z_q(K) \supseteq c_2Z_\infty(K)$, for every $q\gr n$, where $c_2>0$ is an absolute constant.
\end{itemize}
The assertion (a) above is straightforward by the definition of $Z_2(K)$, while (b) is a consequence of reverse H\"{o}lder inequalities for seminorms that hold due to Borell's Lemma \cite{Borell}, see also \cite[Lemma 2.4.5 and Theorem 2.4.6]{BGVV-book}. Fact (c) was first observed by Paouris \cite{Pao2}, see also \cite[Lemma 3.2.8]{BGVV-book}.

The volume of the $L_q$-centroid bodies is an important question, which is not yet completely understood. We collect the known estimates
in the next theorem.

\begin{theorem}\label{thm.vol.Z_q}
Let $K$ be an isotropic convex body in $\mathbb{R}^n$.
\begin{itemize}
\item[\rm(a)] Lutwak, Yang and Zhang have proved in {\rm \cite{LYZ}} that, for every $1\ls q\ls n$,
\begin{equation}\label{eq.Zq.vrad.lower.weak}
\vol_n(Z_q(K))^{1/n} \gtrsim \sqrt{q/n}.
\end{equation}
\item[\rm(b)] Klartag and E.~Milman have proved in {\rm \cite{KM}} that if $q\ls \sqrt{n}$ then the estimate of {\rm(a)} above can be strengthened to
\begin{equation}\label{eq.Zq.vrad.lower}
\vol_n(Z_q(K))^{1/n} \gtrsim \sqrt{q/n}L_K.
\end{equation}
\item[\rm(c)] On the other hand, Paouris has proved in {\rm \cite{Pao3}} that the estimate
\begin{equation}\label{eq.Zq.vrad.upper}
\vol_n(Z_q(K))^{1/n} \lesssim \sqrt{q/n}L_K
\end{equation}
holds for every $1\ls q\ls n$.
\end{itemize}
\end{theorem}


\smallskip

For any isotropic convex body $K$ in $\mathbb{R}^n$ and any $q\neq 0$, $q>-n$, we define
\[
I_q(K) := \left( \int_K \|x\|_2^q\,dx \right)^{1/q}.
\]
Note that $I_2(K) = \sqrt{n}L_K$, since $K$ is isotropic. A direct computation (see \cite[Lemma 3.2.16]{BGVV-book}) shows that
\begin{equation}\label{eq.Iq.formula}
I_q(K) \asymp \sqrt{n/q}\left( \int_{S^{n-1}}\,\int_K |\langle x,\xi\rangle|^q\,dx\,d\sigma(\xi) \right)^{1/q}=\sqrt{n/q}\,w_q(Z_q(K)).
\end{equation}
A similar identity also holds for negative values of $q$; for every $1\ls q < n$,
\begin{equation}\label{eq.w_{-q}.I_{-q}}
w_{-q}(Z_q(K)) \asymp \sqrt{q/n}\,I_{-q}(K).
\end{equation}
This was proved in \cite{Pao4}, see also \cite[Theorem 5.3.16]{BGVV-book}.

An important result of Paouris (see \cite{Pao3} and \cite{Pao4}) states that the quantities $I_q(K)$ remain constant,
of the order of $\sqrt{n}L_K$, as long as $1\ls |q|\ls \sqrt{n}$.

\begin{theorem}[Paouris]\label{thm.paouris.Iq}
Let $K$ be an isotropic convex body in $\mathbb{R}^n$. Then
\[
I_{-q}(K) \asymp I_q(K)\asymp \sqrt{n}L_K,
\]
for every $1\ls q\ls \sqrt{n}$.
\end{theorem}

Theorem \ref{thm.paouris.Iq} implies a very useful large deviation estimate (see \cite{Pao3})
as well as a strong small-ball type inequality (see \cite{Pao4}) for isotropic convex bodies.

\begin{theorem}[Paouris]\label{thm.paouris.smallball}
If $K$ is isotropic in $\mathbb{R}^n$, then
\begin{equation*}\vol_n(\{x\in K : \|x\|_2 \gr c_1t\sqrt{n}L_K\}) \ls e^{-t\sqrt{n}}\end{equation*}
for every $t\gr 1$ and
\begin{equation}\label{eq.paouris.smallball}
\vol_n(\{x\in K : \|x\|_2 \ls \varepsilon\sqrt{n}L_K\}) \ls \varepsilon^{c_2\sqrt{n}}
\end{equation}
for every $0<\varepsilon<\varepsilon_0$, where $\varepsilon_0, c_1,c_2>0$ are absolute constants.
\end{theorem}

\begin{remark}\rm
A useful application of Theorem \ref{thm.paouris.Iq} is the next estimate for the mean width of $Z_q(K)$, when $q\lesssim \sqrt{n}$.
If $K$ is an isotropic convex body in $\mathbb{R}^n$ then, for every $1\ls q \ls \sqrt{n}$,
\begin{equation}\label{eq.Zq.mean.width}
w(Z_q(K)) \asymp \sqrt{q}L_K.
\end{equation}
This estimate is a standard consequence of the results of Paouris in \cite{Pao3}: note that
$$w_q(Z_q(K)) \asymp \sqrt{q/n} I_q(K) \asymp \sqrt{q/n} I_2(K) = \sqrt{q}L_K.$$
Since $w(Z_q(K))\ls w_q(Z_q(K))$, by H\"{o}lder's inequality, we see that $w(Z_q(K))\lesssim \sqrt{q}L_K$.
For the reverse inequality we use the estimate on the volume of $Z_q(K)$ (Theorem \ref{thm.vol.Z_q} (b)), and Urysohn's inequality to write
\[
w(Z_q(K)) \gr \left(\frac{\vol_n(Z_q(K))}{\omega_n}\right)^{1/n} \gtrsim \sqrt{n}\sqrt{q/n}L_K.
\]
\end{remark}

\section{Random convex hulls in isotropic convex bodies}

Let $N\gr n$ and $x_1,\ldots,x_N$ be independent random vectors chosen uniformly from an isotropic convex body $K$ in $\mathbb{R}^n$. Consider the symmetric random polytope
$$K_N:=\conv\{\pm x_1,\ldots, \pm x_N\}.$$
The next two facts were proved in \cite{DGT} and \cite[Lemma~3.1]{GHiT}:
\begin{enumerate}
\item[({\bf P1})] There exist absolute constants $\alpha ,c_1,c_2,c_3>0$ such that if $N\gr\alpha n$ and $q\ls c_1\log (N/n)$ then the inclusion
\begin{equation}\label{eq:P1}
K_N={\rm conv}(\{ \pm x_1,\ldots ,\pm x_N\})\supseteq c_2 Z_q(K)
\end{equation}
holds with probability greater than $1-e^{-c_3\sqrt{N}}$.
\item[({\bf P2})] For any $q\gr\log N$ and $t\gr 1$, the inequality
\begin{equation}\label{eq:P2}
w(K_N)\ls c_3tw (Z_q(K))
\end{equation}
holds with probability greater than $1-t^{-q}$.
\end{enumerate}
Combining these basic asymptotic properties of a random $K_N$ with the results of the previous section we get:

\begin{theorem}[Dafnis-Giannopoulos-Tsolomitis]\label{thm.vrad.mean.width.general}
Let $n, N \in \mathbb{N}$, and $K$ be an isotropic convex body in $\mathbb{R}^n$.
\begin{itemize}
\item[\rm(a)] If $n\lesssim N\ls e^{\sqrt{n}}$, then
\begin{equation}\label{eq:vol-lower-bound}
\vol_n(K_N)^{1/n} \gtrsim \sqrt{\log(2N/n)/n}L_K
\end{equation}
with probability greater than $1-\exp(-c\sqrt{N})$ for some absolute constant $c>0$.
\item[\rm(b)] If $n\lesssim N\ls e^{\sqrt{n}}$, then for every $1\ls k \ls n$ we have
\[
\sqrt{\log(2N/n)}L_K \lesssim Q_k(K_N)\ls w(K_N) \lesssim \sqrt{\log N}L_K
\]
with probability greater than $1-\frac{1}{N}$.
\end{itemize}
\end{theorem}

For a proof of all these assertions see \cite{DGT}, \cite{DGT2}, and also \cite[Chapter~11]{BGVV-book}.
Moreover, in the range $n\lesssim N\ls e^{\sqrt{n}}$, one can further check that an upper bound of the order $\sqrt{\log N}L_K$ holds for the
volume radius of a random $k$-dimensional projection of a random $K_N$ (see \cite[Fact~4.6]{DGT2}). Starting from the inequality
\begin{equation*}Q_k(K_N)=\left (\frac{1}{\omega_k}\int_{G_{n,k}}\vol_k(P_F(K_N))\,d\nu_{n,k}(F)\right
)^{1/k}\lesssim \sqrt{\log N}L_K\end{equation*} and applying Markov's inequality, we get:

\begin{lemma}\label{lem:random-projection}If $n\lesssim N\ls e^{\sqrt{n}}$ then with probability greater than $1-\frac{1}{N}$ the random polytope $K_N$
satisfies the following: for every $1\ls k\ls n$ and $t>1$,
\begin{equation}\label{eq.PFKN.upper}
\nu_{n,k}\Big(\Big\{ F\in G_{n,k} : \Big(\frac{\vol_k(P_F(K_N))}{\omega_k} \Big)^{1/k} \ls c_1t\sqrt{\log N}L_K \Big\}\Big) \gr 1-t^{-k}.
\end{equation}
\end{lemma}

These estimates suffice for a proof of Theorem~\ref{th:random}.

\begin{proof}[Proof of Theorem~\ref{th:random}]
From Theorem~\ref{thm.vrad.mean.width.general} and Lemma~\ref{lem:random-projection} we know that with probability greater than $1-\frac{1}{N}-e^{-c\sqrt{N}}$,
the random polytope $K_N$ satisfies the volume bound
\begin{equation}\label{eq:vol-bound}\vol_n(K_N)^{1/n} \gtrsim \sqrt{\log(2N/n)/n}L_K\end{equation}
and also $\nu_{n,k}(A)\gr 1-2^{-k}$, where $A=\left\{ F\in G_{n,k} : \left(\frac{\vol_k(P_F(K_N))}{\omega_k} \right)^{1/k} \ls 2c_1\sqrt{\log N}L_K \right\}$. Therefore,
\begin{align*}
\int_{G_{n,k}} \vol_k(P_F(K_N))^{-n}\,d\nu_{n,k}(F) &\gr \int_A \vol_k(P_F(K_N))^{-n}\,d\nu_{n,k}(F) \\
                                                  &\gr (1-2^{-k})(2c_1\sqrt{\log N} \omega_k^{1/k}L_K)^{-kn} \\
                                                  &\gr (4c_1\sqrt{\log N}\omega_k^{1/k}L_K)^{-kn}.
\end{align*}
It follows that, with probability greater than $1-\frac{2}{N}$, we have that for every $1\ls k\ls n$,
\begin{equation}\label{eq.int.PFKN^-n.upper}
\left(\int_{G_{n,k}} \vol_k(P_F(K_N))^{-n}\,d\nu_{n,k}(F) \right)^{-\frac{1}{kn}} \ls c_2\sqrt{\log N}\omega_k^{1/k}L_K.
\end{equation}
Combining with \eqref{eq:vol-bound} we write
\begin{align*}
\Phi_{[k]}(K_N) &= \vol_n(K_N)^{-\frac{1}{n}}\left (\int_{G_{n,k}}\vol_k(P_F(K_N))^{-n}\,d\nu_{n,k}(F)\right )^{-\frac{1}{kn}}\\
&\lesssim \frac{\sqrt{\log N}}{\sqrt{\log(N/n)}} \frac{\sqrt{n}}{\omega_k^{-1/k}} \lesssim \sqrt{n/k},
\end{align*}
since $\omega_k^{-1/k}\asymp \sqrt{k}$ and $\log N \ls 2\log(N/n)$ (because $N\gr n^2$).
\end{proof}

\section{Random polytopes with vertices on convex surfaces}

We assume that $K$ is an isotropic convex body in ${\mathbb R}^n$. Recall that the cone probability measure $\mu_K$ on the
boundary $\partial (K)$ of $K$ is defined by
$$\mu_K(B)=\frac{\vol_n(\{rx:x\in B,0\ls r\ls 1\})}{\vol_n(K)}$$
for all Borel subsets $B$ of $\partial (K)$. For any $N>n$ we consider independent random points $x_1,\ldots ,x_N$ distributed according to $\mu_K$
and the random polytope $M_N={\rm conv}\{\pm x_1,\ldots ,\pm x_N\}$. We can describe the asymptotic shape of $M_N$ with some
modifications of the approach of \cite{DGT}. We start with the next inclusion lemma.

\begin{lemma}\label{lem:P1-cone}
There exist absolute constants $\alpha ,c_1,c_2,c_3>0$ such that if $N\gr\alpha n$ and $q\ls c_1\log (N/n)$ then the inclusion
\begin{equation}\label{eq:P1-cone}
M_N={\rm conv}(\{ \pm x_1,\ldots ,\pm x_N\})\supseteq c_2 Z_q(K)
\end{equation}
holds with probability greater than $1-e^{-c_3\sqrt{N}}$.
\end{lemma}

\begin{proof}We sketch the argument from \cite{Horrmann-Prochno-Thale-2017}. Consider $N$ independent random points $y_1,\ldots ,y_N$
with distribution $\lambda_K$. We define $N$ points $x_1,\ldots ,x_N\in\partial (K)$ as follows: if $y_i\neq 0$ for all $1\ls i\ls N$
then we set $x_i=y_i/\|y_i\|_K$. In all other cases we choose $x_1=\cdots =x_N=u$, where $u$ is an arbitrary point in $\partial (K)$.
Note that for every Borel subset $B$ of $\partial (K)$ we have
$$\lambda_K(\{ y\in K:y/\|y\|_K\in B\})=\mu_K(B),$$
which means that the independent random points $x_1,\ldots ,x_N$ are distributed according to the cone measure $\mu_K$. Therefore,
the distribution of ${\rm conv}\{\pm x_1,\ldots ,\pm x_N\}$ is exactly the same as the distribution of $M_N$. Moreover, we have
$${\rm conv}\Big\{\pm \frac{y_1}{\|y_1\|_K},\ldots ,\pm \frac{y_N}{\|y_N\|_K}\Big\}\supseteq {\rm conv}\{\pm y_1,\ldots ,\pm y_N\}=K_N$$
with probability $1$. Then, the lemma follows from $({\bf P1})$.
\end{proof}

\begin{lemma}\label{lem:P2-cone}
If $n\lesssim N\ls e^{\sqrt{n}}$ and $q_0=2\log (2N)$, then the inequality
\begin{equation}\label{eq:P2-cone}
w(M_N)\lesssim w_{q_0}(Z_{q_0}(K))\asymp \sqrt{\log N}L_K
\end{equation}
holds with probability greater than $1-\frac{1}{4N^2}$.
\end{lemma}

\begin{proof}
Let $\xi \in S^{n-1}$. If $X$ is a random vector distributed according to $\mu_K$ then, for any $t>1$ we have
\[
\mathbb{P}(|\langle X,\xi\rangle| \gr t \|\langle\cdot,\xi\rangle\|_{L^q(\mu_K)}) \ls t^{-q},
\]
by Markov's inequality. Therefore,
\begin{equation}\label{eq:weak-reverse}
\mathbb{P}(h_{M_N}(\xi) \gr t\|\langle\cdot,\xi\rangle\|_{L^q(\mu_K)}) = \mathbb{P}\left(\max_{j\ls N} |\langle x_j,\xi\rangle| \gr t\|\langle\cdot,\xi\rangle\|_{L^q(\mu_K)}\right) \ls Nt^{-q}.\end{equation}
Using the identity
\begin{equation*}\int_{{\mathbb R}^n}f(x)\,dx=n\,{\rm vol}_n(K)\int_0^{\infty }r^{n-1}\int_{\partial (K)}f(rx)\,d\mu_K(x)\,dr\end{equation*}
which holds for every integrable function $f:{\mathbb R}^n\to {\mathbb R}$ (see \cite[Proposition~1]{Naor-Romik}) one can check that
$$\int_K|\langle x,\xi\rangle |^qdx=\frac{n}{n+q}\int_{\partial (K)}|\langle x,\xi\rangle |^qd\mu_K(x)$$
for every $q\gr 0$; the computation can be found in \cite[Lemma~3.2]{Prochno-Thale-Turchi}. Equivalently, we may write
\begin{equation}\label{eq.P2-cone.LqmuK-ZqK}
\frac{n}{n+q}\|\langle \cdot ,\xi\rangle\|^q_{L^q(\mu_K)}= h_{Z_q(K)}(\xi )^q.
\end{equation}
Since $M_N \subseteq R(K)B_2^n\subseteq c_2nL_KB_2^n$ and $Z_q(K)\supseteq Z_2(K) \asymp L_KB_2^n$, we have
$$h_{M_N}(\xi ) \ls c_1nh_{Z_q(K)}(\xi )\ls c_1n\|\langle \cdot ,\xi\rangle\|_{L^q(\mu_K)}$$ for every $\xi \in S^{n-1}$. Therefore,
\[
\int_{S^{n-1}} \left( \frac{h_{M_N}(\xi )}{\|\langle \cdot ,\xi\rangle\|_{L^q(\mu_K)}} \right)^q \,d\sigma(\xi)
= \int_0^{c_1n} qt^{q-1} \sigma\left( \{\xi \in S^{n-1} : h_{M_N}(\xi ) \gr t\|\langle \cdot ,\xi\rangle\|_{L^q(\mu_K)}\} \right)\,dt.
\]
Taking expectations and using \eqref{eq:weak-reverse} we get, for every $\alpha >1$,
\[
\mathbb{E}\left( \int_{S^{n-1}} \frac{h_{M_N}(\xi )^q}{\|\langle \cdot ,\xi\rangle\|_{L^q(\mu_K)}^q} \,d\sigma(\xi) \right)
\ls \alpha^q +\int_\alpha^{c_1n} qt^{q-1}Nt^{-q}\,dt = \alpha^q+qN\log\left( \frac{c_1n}{\alpha}\right).
\]
Note that the choice $q_0:=2\log(2N)$ implies $e^{q_0}=(2N)^2 \gtrsim q_0N\log\left( \frac{c_1n}{2e}\right)$, so applying the above for $\alpha=2e$ we get
\[
\mathbb{E}\left( \int_{S^{n-1}} \frac{h_{M_N}(\xi )^{q_0}}{\|\langle \cdot ,\xi\rangle\|_{L^{q_0}(\mu_K)}^{q_0}} \,d\sigma(\xi) \right) \ls  c_2^{q_0},
\]
where $c_2>0$ is an absolute constant. Then by Markov's inequality we get that
\begin{equation}\label{eq.P2-cone.Markov}
\int_{S^{n-1}} \frac{h_{M_N}(\xi )^{q_0}}{\|\langle \cdot ,\xi\rangle\|_{L^{q_0}(\mu_K)}^{q_0}} \,d\sigma(\xi) \ls (c_2e)^{q_0}
\end{equation}
with probability greater than $1-e^{-q_0}=1-\frac{1}{4N^2}$. Now, using successively H\"{o}lder's inequality, the Cauchy-Schwarz
inequality, \eqref{eq.P2-cone.LqmuK-ZqK} and \eqref{eq.P2-cone.Markov}, we write
\begin{align*}
w(M_N)^{q_0} &\ls \Big (\int_{S^{n-1}}h_{M_N}(\xi )^{q_0/2}d\sigma (\xi )\Big )^2\ls w_{q_0}(Z_{q_0}(K))^{q_0}\int_{S^{n-1}}\frac{h_{M_N}(\xi )^{q_0}}{h_{Z_{q_0}(K)}(\xi )^{q_0}}d\sigma (\xi )\\
&\ls \frac{n+q_0}{n}w_{q_0}(Z_{q_0}(K))^{q_0}\int_{S^{n-1}}\frac{h_{M_N}(\xi )^{q_0}}{\|\langle \cdot ,\xi\rangle\|_{L^{q_0}(\mu_K)}^{q_0}}d\sigma (\xi )\\
&\ls 2w_{q_0}(Z_{q_0}(K))^{q_0}(c_2e)^{q_0},
\end{align*}
and we conclude that
$$w(M_N)\lesssim w_{q_0}(Z_{q_0}(K))\lesssim \sqrt{q_0}L_K\asymp \sqrt{\log N}L_K$$
with probability greater than $1-\frac{1}{4N^2}$, taking into account \eqref{eq.Iq.formula} and our choice of $q_0$.
\end{proof}

These two lemmas establish the analogues of $({\bf P1})$ and $({\bf P2})$ in the case of $M_N$. Then, as with $K_N$, we can immediately conclude
the following.

\begin{theorem}\label{th:basic-estimates-cone}
Let $n, N \in \mathbb{N}$, and $K$ be an isotropic convex body in $\mathbb{R}^n$.
\begin{itemize}
\item[\rm(a)] If $n\lesssim N\ls e^{\sqrt{n}}$, then
\[
\vol_n(M_N)^{1/n} \gtrsim \sqrt{\log(2N/n)/n}L_K
\]
with probability greater than $1-\exp(-c\sqrt{N})$ for some absolute constant $c>0$.
\item[\rm(b)] If $n\lesssim N\ls e^{\sqrt{n}}$, then for every $1\ls k \ls n$ we have
\[
\sqrt{\log(2N/n)}L_K \lesssim Q_k(M_N)\ls w(M_N) \lesssim \sqrt{\log N}L_K
\]
with probability greater than $1-\frac{1}{4N^2}$.
\end{itemize}
\end{theorem}

Having proved Theorem~\ref{th:basic-estimates-cone} we can repeat the proof of Theorem \ref{th:random} to get Theorem~\ref{th:random-cone}.

\smallskip

We conclude this section with a proof of an upper bound for the volume radius of a random $M_N$.

\begin{theorem}\label{thm.vrad.upper}Let $K$ be an isotropic body in $\mathbb{R}^n$. If $n\lesssim N\ls e^n$, then
\[
\vol_n(M_N)^{1/n} \ls c \sqrt{\log N/n}\,L_K,
\]
with probability greater than $1-\frac{1}{4N^2}$, where $c>0$ is an absolute constant.
\end{theorem}

\begin{proof}We fix ${q_0}=2\log (2N)$ and check that
\begin{equation}\label{eq:vrad-upper-1}
w_{-{q_0}}(M_N)  \lesssim w_{-{q_0}/2}(Z_{{q_0}}(K)),
\end{equation}
with probability greater than $1-\frac{1}{4N^2}$. To see this, we write
\begin{align}\label{eq.w{-q}(Zq).lem}
(w_{-{q_0}/2}(Z_{q_0}(K)))^{-{q_0}} &= \left( \int_{S^{n-1}} \frac{1}{h_{Z_{q_0}(K)}(\xi)^{{q_0}/2}} \,d\sigma(\xi)\right)^2 \nonumber\\
                        &\ls \left( \int_{S^{n-1}} \frac{1}{h_{M_N}(\xi)^{q_0}}\,d\sigma(\xi) \right) \left( \int_{S^{n-1}} \frac{h_{M_N}(\xi)^{q_0}}{h_{Z_{q_0}(K)}(\xi)^{q_0}} \,d\sigma(\xi)\right)\nonumber\\
                        &= w_{-{q_0}}(M_N)^{-{q_0}} \left( \int_{S^{n-1}} \frac{h_{M_N}(\xi)^{q_0}}{h_{Z_{q_0}(K)}(\xi)^{q_0}} \,d\sigma(\xi)\right).
\end{align}
In the proof of Lemma~\ref{lem:P2-cone} we saw that
\[
\int_{S^{n-1}} \frac{h_{M_N}(\xi )^{q_0}}{h_{Z_{q_0}(K)}(\xi)^{q_0}} \,d\sigma(\xi) \ls (c_2e)^{q_0}
\]
with probability greater than $1-e^{-q_0}=1-\frac{1}{4N^2}$. Combining the above we get \eqref{eq:vrad-upper-1}.

Recall that, for any symmetric convex body $A$ in ${\mathbb R}^n$ and $q\ls n$,
\[
\left(\frac{\vol_n(A^\circ)}{\omega_n} \right)^{1/n} = \left( \int_{S^{n-1}} h_A(\xi)^{-n}\,d\sigma(\xi) \right)^{1/n}
\gr \left( \int_{S^{n-1}} h_A(\xi)^{-q}\,d\sigma(\xi) \right)^{1/q} = \frac{1}{w_{-q}(A)}.
\]
Using the Blaschke-Santal\'{o} inequality and the fact that $\omega_n^{1/n} \asymp 1/\sqrt{n}$ we get
\begin{equation}\label{eq.vol.bound.w_{-q}}
\vol_n(A)^{1/n} \ls \omega_n^{2/n}\vol_n(A^\circ)^{-1/n} \ls \omega_n^{1/n}w_{-q}(A) \ls c_1 \frac{w_{-q}(A)}{\sqrt{n}},
\end{equation}
for some absolute constant $c_1>0$.

Using successively \eqref{eq.vol.bound.w_{-q}} and \eqref{eq:vrad-upper-1}, we get
\[
\vol_n(M_N)^{1/n} \ls c_1 \frac{w_{-{q_0}}(M_N)}{\sqrt{n}}\lesssim \frac{w_{-{q_0}/2}(Z_{{q_0}}(K))}{\sqrt{n}}.
\]
Since $Z_{q_0}(K)\subseteq cZ_{{q_0}/2}(K)$, we have
\[
\vol_n(M_N)^{1/n} \lesssim \frac{w_{-{q_0}/2}(Z_{{q_0}/2}(K))}{\sqrt{n}}\asymp \frac{\sqrt{{q_0}}}{n}I_{-{q_0}/2}(K),
\]
taking into account \eqref{eq.w_{-q}.I_{-q}}. Finally, since $I_{-q/2}\ls I_2(K) = \sqrt{n}L_K$ is valid for any $q\ls n$, we get
\[
\vol_n(M_N)^{1/n} \lesssim \frac{\sqrt{q_0}}{\sqrt{n}}L_K \asymp \frac{\sqrt{\log 2N}}{\sqrt{n}}L_K,
\]
with probability greater than $1-e^{-q_0}=1-\frac{1}{4N^2}$.
\end{proof}

\section{Beta polytopes}

Recall that, for $\beta>-1$, $\nu_\beta$ is the probability measure supported on $B_2^n$, with density
\begin{equation*}
p_{n,\beta}(x) = c_{n,\beta}(1-\|x\|_2^2)^\beta,
\end{equation*}
where
\[
c_{n,\beta} := \pi^{-n/2}\frac{\Gamma\left(\beta+\frac{n}{2}+1 \right)}{\Gamma(\beta+1)}.
\]
The one-dimensional marginal density of $\nu_\beta$ is given by
\begin{equation*}
f_{\beta}(t) = \alpha_{n,\beta}(1-t^2)^{\beta+\frac{n-1}{2}}, \quad t\in[-1,1],
\end{equation*}
where $\alpha_{n,\beta} := c_{n,\beta}/c_{n-1,\beta}$. For $d\in[0,1]$, let
\[
\B(d) := \int_d^1 f_{\beta}(t)\,dt.
\]
Note that $\B(0)=1/2$, $\B(1)=0$ and $\B$ is a decreasing function of $d$. We will use the following bounds on $\B(d)$, originally established in \cite[Lemma~2.2]{BCGTT}. For any $d\in (0,1)$,
\begin{equation}\label{eq:B.bounds}
\frac{1}{2\sqrt{\pi}}\frac{(1-d^2)^{\beta+\frac{n+1}{2}}}{\sqrt{\beta+\frac{n}{2}+1}} \ls \B(d) \ls \frac{1}{2d\sqrt{\pi}}\frac{(1-d^2)^{\beta+\frac{n+1}{2}}}{\sqrt{\beta+\frac{n}{2}}}.
\end{equation}
Let $N>n$, and $x_1,\ldots,x_N$ be random vectors, chosen independently according to the measure $\nu_\beta$. Let
\[
P_{N,n}^{\beta} := \conv\{x_1,\ldots,x_N\}.
\]
We will refer to this random convex hull as the beta polytope (with parameter $\beta>-1$) in $\mathbb{R}^n$.

\smallskip

In this section, we prove Theorem \ref{th:random-beta}.
The statement will follow from the next two lemmas.

\begin{lemma}\label{lem:beta.lower}
Let $\beta>-1$ and $N>n$. Then if $g(n,\beta) := 2\sqrt{\pi}n\sqrt{\beta+\frac{n}{2}+1}(1+\log(4\sqrt{\beta+\frac{n}{2}+1}))$ and $N\gr c_1g(n,\beta)$,
\[
\vol_n(P_{N,n}^{\beta})^{1/n}\gr c\frac{\sqrt{1-(g(n,\beta)/N)^{\frac{2}{2\beta+n+1}}}}{\sqrt{n}}.
\]
holds with probability greater than $1-e^{-n}$.
\end{lemma}

\begin{proof}
For any $R>\frac{1}{2}\sqrt{1-(g(n,\beta)/N)^{\frac{2}{2\beta+n+1}}}$, note that
\[
\mathbb{P}\left(\frac{1}{2}\sqrt{1-(g(n,\beta)/N)^{\frac{2}{2\beta+n+1}}}B_2^n\nsubseteq P_{N,n}^\beta\right) \ls \mathbb{P}(RB_2^n\nsubseteq P_{Nn}^\beta)
\]
so the statement of the Lemma will follow, once we prove that
\[
\mathbb{P}(RB_2^n\nsubseteq P_{N,n}^\beta) \ls e^{-n},
\]
for a suitable value of $R$.

Fix some $\varepsilon\in (0,1)$ to be determined, and let ${\cal N}_\varepsilon$ be an $\varepsilon$-net on $S^{n-1}$, of cardinality $|{\cal N}_\varepsilon|\ls \left(\frac{2}{\varepsilon}\right)^n$. Note that, for any $x\in B_2^n$ and $a>0$, if $\langle x,\xi \rangle \ls a$ holds for some $\xi\in S^{n-1}$, then $\langle x,\xi \rangle \ls a+\varepsilon$ holds for some $\xi \in {\cal N}_\varepsilon$. Using the union bound and the independence of the vertices $X_i$, we can then write
\begin{align*}
\mathbb{P}(RB_2^n\nsubseteq P_{N,n}^\beta) &\ls \mathbb{P}\left(h_{P_{N,n}^\beta}(\xi)< R, \hbox{ for some } \xi\in S^{n-1}\right)\\
                                           &\ls \mathbb{P}\left(h_{P_{N,n}^\beta}(\xi)< R+\varepsilon , \hbox{ for some } \xi\in {\cal N}_\varepsilon\right) \ls \left(\frac{2}{\varepsilon} \right)^n \mathbb{P}\left(\max_{1\ls i\ls N}\langle X_i,\xi\rangle < R+\varepsilon\right)\\
                                           &=  \left(\frac{2}{\varepsilon} \right)^n \left(\mathbb{P}(\langle X,\xi\rangle < R+\varepsilon)\right)^N.
\end{align*}
Next note that, for any $\xi\in S^{n-1}$ and $d\in(0,1)$, due to the rotational invariance of $\nu_\beta$,
\begin{align*}
\mathbb{P}(\langle X,\xi\rangle < d) &= \mathbb{P}(\langle X,e_1\rangle < d) = c_{n,\beta} \int_{-1}^d \int_{\sqrt{1-x_1^2}B_2^{n-1}} (1-\|x\|_2^2)^\beta \,d(x_2,\ldots,x_n)\,dx_1 \\
                                          &= c_{n,\beta} \int_{-1}^d \int_{\sqrt{1-x_1^2}B_2^{n-1}} (1-x_1^2)^\beta \left(1-\frac{\sum_{i=2}^n x_i^2}{1-x_1^2}\right)^\beta \,d(x_2,\ldots,x_n)\,dx_1 \\
                                          &= c_{n,\beta} \int_{-1}^d (1-t^2)^\beta \int_{B_2^{n-1}} (1-\|z\|_2^2)^\beta (1-t^2)^{\frac{n-1}{2}} \,dz\,dt \\
                                          &= \alpha_{n,\beta} \int_{-1}^d (1-t^2)^{\beta+\frac{n-1}{2}}\,dt \int_{B_2^{n-1}} p_{n-1,\beta}(z)\,dz \\
                                          &= \int_{-1}^d f_\beta(t)\,dt = 1-\B(d).
\end{align*}
Combining the above, we get
\begin{align*}
\mathbb{P}(RB_2^n\nsubseteq P_{N,n}^\beta) &\ls \left(\frac{2}{\varepsilon} \right)^n (1-\B(R+\varepsilon))^N\\
                                           &\ls \left(\frac{2}{\varepsilon} \right)^n \exp(-N\B(R+\varepsilon)) = \exp\left(n\log(2/\varepsilon) - N\B(R+\varepsilon) \right),
\end{align*}
so we need to prove that $n\log(2/\varepsilon) - N\B(R+\varepsilon) \ls -n$, or, equivalently,
\[
N\B(R+\varepsilon) \gr n\left(1+\log(2/\varepsilon)\right).
\]
Now let $\varepsilon = (2\sqrt{\beta+\frac{n+1}{2}})^{-1}$, and note that if we choose $N>g(n,\beta)\left(1-4\varepsilon^2 \right)^{-\left(\beta+\frac{n+1}{2}\right)}$ (which is satisfied if $N\gr c_1g(n,\beta)$ for an absolute constant $c_1>0$),
it follows that $\varepsilon < \frac{1}{2}\sqrt{1-(g(n,\beta)/N)^{\frac{2}{2\beta+n+1}}}$.
Taking $R = \sqrt{1-(g(n,\beta)/N)^{\frac{2}{2\beta+n+1}}} -\varepsilon$ and using the lower bound in \eqref{eq:B.bounds}, we can see then that
\[
N\B(R+\varepsilon) \gr \frac{N}{2\sqrt{\pi}}\frac{(g(n,\beta)/N)^{\frac{2}{2\beta+n+1}\left(\beta+\frac{n+1}{2}\right)}}{\sqrt{\beta+\frac{n}{2}+1}}
      = \frac{g(n,\beta)}{2\sqrt{\pi}\sqrt{\beta+\frac{n}{2}+1}}  \gr n\left(1+\log(2/\varepsilon)\right),
\]
which completes the proof.
\end{proof}

\smallskip

The average, on $G_{n,k}$, of the volume of $P_F(P_{N,n}^{\beta})$ is related to the volume of $P_{N,k}^{\beta'}$, for some $\beta'>-1$, as follows: Recall that, for $k=1,\ldots,n$, the $k$-th intrinsic volume of a convex body $K$ in $\mathbb{R}^n$ has an integral representation, given by Kubota's formula,
\[
V_k(K) = r_{n,k}\int_{G_{n,k}} \vol_k(P_F(K)) \,d\nu_{n,k}(F),
\]
where $r_{n,k} = \binom{n}{k}\frac{\omega_n}{\omega_{k}\omega_{n-k}}$. It is known (see \cite[Proposition~2.3]{KTT}) that
\[
\mathbb{E}(V_k(P_{N,n}^\beta)) = r_{n,k}\mathbb{E}(\vol_k(P_{N,k}^{\beta+\frac{n-k}{2}})).
\]
It follows, that
\begin{equation}\label{eq.beta.avg.k-dim.vol}
\mathbb{E}\,\Big (\int_{G_{n,k}} \vol_k(P_F(P_{N,n}^\beta)) \,d\nu_{n,k}(F)\Big) = \mathbb{E}(\vol_k(P_{N,k}^{\beta+\frac{n-k}{2}})).
\end{equation}
Using this fact, we prove Lemma \ref{lem:beta.upper}.

\begin{lemma}\label{lem:beta.upper}
For any $\beta>-1$ and $k\gr \log\left(n\left(1+\log\left(4\sqrt{\beta+\frac{n}{2}+1}\right)\right)\right)$, if $N\gr g(n,\beta)c_1^{\beta+\frac{n+1}{2}}$, then the event
\[
\nu_{n,k}\Big(\Big\{F\in G_{n,k} : \vol_k(P_F(P_{N,n}^{\beta}))^{1/k} \ls c_2\frac{1}{\sqrt{k}}\sqrt{1-(g(n,\beta)/N)^{\frac{2}{2\beta+n+1}}}\Big\}\Big) \gr 1-e^{-k},
\]
holds, with probability greater than $1-e^{-k}$.
\end{lemma}

\begin{proof}
Let $r:=\sqrt{1-(g(n,\beta)/N)^{\frac{2}{2\beta+n+1}}}$, and $S_t = \left\{F\in G_{n,k} : \vol_k(P_F(P_{N,n}^{\beta}))^{1/k} \gr \frac{tr}{\sqrt{k}}\right\}$. By Markov's inequality,
\[
\nu_{n,k}(S_t) \ls \left(\frac{\sqrt{k}}{tr}\right)^k\int_{G_{n,k}} \vol_k(P_F(P_{N,n}^\beta)) \,d\nu_{n,k}(F).
\]
We will prove that there is an absolute constant $c_3>0$ such that
\[
\mathbb{P}\left( \int_{G_{n,k}} \vol_k(P_F(P_{N,n}^\beta)) \,d\nu_{n,k}(F) \ls \left(c_3\frac{r}{\sqrt{k}}\right)^k\right) \gr  1-e^{-k},
\]
which implies the statement of the lemma (with $c_2=c_3e$) if we choose $t=c_3e$. Note that by Markov's inequality again, there exists
an absolute constant $c_4>0$ such that, applying also \eqref{eq.beta.avg.k-dim.vol},
\begin{align*}
\mathbb{P}\left( \int_{G_{n,k}} \vol_k(P_F(P_{N,n}^\beta)) \,d\nu_{n,k}(F) \gr \left(\frac{c_4er}{\sqrt{k}}\right)^k\right)
&\ls e^{-k} \mathbb{E}\left (\frac{\int_{G_{n,k}} \vol_k(P_F(P_{N,n}^\beta))\,d\nu_{n,k}(F)}{\vol_k\left(rB_2^k\right)}\right)\\
&= e^{-k} \mathbb{E}\left (\frac{\vol_k(P_{N,k}^{\beta+\frac{n-k}{2}})}{\vol_k\left(rB_2^k\right)}\right),
\end{align*}
so the problem is reduced to establishing a correct upper bound for $\mathbb{E}(\vol_k(P_{N,k}^{\beta+\frac{n-k}{2}}))$. We will prove that
\[
\mathbb{E}\left (\frac{\vol_k(P_{N,k}^{\beta+\frac{n-k}{2}})}{\vol_k\left(rB_2^k\right)}\right) \ls c_5^k,
\]
for some absolute constant $c_5>0$.

We will use the fact, proved in \cite[Lemma 3.3 (a)]{BCGTT}, that for every $\beta>-1$, $m\in \mathbb{N}$, and any bounded $A\subset \mathbb{R}^m$,
\[
\mathbb{E}(\vol_m(P_{N,m}^\beta\cap A)) \ls N\sup_{x\in A}\B(\|x\|_2)\vol_m(A).
\]
If $A=B_2^k\setminus rB_2^k$, then, since $\|x\|_2\gr r$ for every $x\in A$,
\begin{align*}
\mathbb{E}\left (\frac{\vol_k(P_{N,k}^{\beta+\frac{n-k}{2}})}{\vol_k(rB_2^k)}\right) &\ls 1 + \mathbb{E}\left(\frac{\vol_k(P_{N,k}^{\beta+\frac{n-k}{2}}\cap A)}{\vol_k(rB_2^k)}\right) \ls 1+r^{-k}N\B(r).
\end{align*}
Note that the hypothesis on $N$ implies that $r\gr c_6$ for some absolute constant $c_6>0$. Using the upper bound in \eqref{eq:B.bounds} we get, for $k\gr \log\left(n\log\left(\sqrt{\beta+\frac{n}{2}+1}\right)\right)$,
\begin{equation*}
r^{-k}N\B(r) \ls N \frac{1}{2\sqrt{\pi}\sqrt{\beta+\frac{n}{2}}} \frac{(g(n,\beta)/N)^{\frac{2}{2\beta+n+1}\left(\beta+\frac{n+1}{2}\right)}}{r^{k+1}}
\ls \frac{g(n,\beta )}{2\sqrt{\pi}\sqrt{\beta+\frac{n}{2}}}\,\frac{1}{c_6^{k+1}}\ls c^k,
\end{equation*}
proving the desired result.
\end{proof}

\smallskip

It is now clear that, having proved Lemma~\ref{lem:beta.lower} and Lemma~\ref{lem:beta.upper}, we can conclude Theorem~\ref{th:random-beta}
exactly as we did for Theorem~\ref{th:random}.

\section{The unconditional case}\label{sec:unco}

Let $K$ be an unconditional convex body in ${\mathbb R}^n$ (this means that $K$ has a linear image that is symmetric with respect to the coordinate
subspaces $e_i^{\perp }$, where $\{e_1,\ldots ,e_n\}$ is an orthonormal basis of ${\mathbb R}^n$). Since the quantity $\Phi_{[k]}(K)$ is linearly invariant,
we may assume that $K$ is in the isotropic position. Then, we may assume that $K$ is symmetric with respect to the coordinate
subspaces and from a well-known result of Bobkov and Nazarov (see \cite{Bobkov-Nazarov-2003a} and \cite{Bobkov-Nazarov-2003b}) we have
\begin{equation*}c_1B_{\infty }^n\subseteq K\subseteq c_2nB_1^n\end{equation*}
for some absolute constants $c_1,c_2>0$, and hence the problem can be reduced to the question to give precise estimates for $\Phi_{[k]}(K)$
in the case where $K$ is the cross-polytope $B_1^n={\rm conv}\{\pm e_1,\ldots ,\pm e_n\}$. Indeed, we have
$\vol_k(P_F(K))^{\frac{1}{k}}\ls c_2n\,\vol_k(P_F(B_1^n))^{\frac{1}{k}}$ for every $F\in G_{n,k}$, and hence
\begin{align*}W_{[k,-p]}(K)&=\left (\int_{G_{n,k}}\vol_k(P_F(K))^{-p}\,d\nu_{n,k}(F)\right )^{-\frac{1}{kp}}\\
&\ls c_2n\left (\int_{G_{n,k}}\vol_k(P_F(B_1^n))^{-p}\,d\nu_{n,k}(F)\right )^{-\frac{1}{kp}}\ls c_3W_{[k,-p]}(B_1^n)\end{align*}
for every $p\neq 0$ and $1\ls k\ls n-1$, if we recall that $\vol_n(B_1^n)=\frac{2^n}{n!}$, and hence
$\vol_n(B_1^n)^{-1/n}\asymp n$.

We proceed to examine the case of the cross-polytope $B_1^n$. A first observation is that $B_1^n\supseteq \frac{1}{\sqrt{n}}B_2^n$ and then, clearly,
$P_F(B_1^n)\supseteq \frac{1}{\sqrt{n}}B_2^n\cap F$ for every $1\ls k\ls n-1$ and $F\in G_{n,k}$. This implies that
\begin{equation*}\vol_k(P_F(B_1^n))^{\frac{1}{k}}\gr \frac{\omega_k^{1/k}}{\sqrt{n}}\asymp \frac{1}{\sqrt{kn}}.\end{equation*}
On the other hand, it is known that $\vol_k(B_\infty^n\cap F)\gr \vol_k(B_\infty^k)=2^k$ and, since $B_\infty^n\cap F$ is the polar body of $P_F(B_1^n)$ in $F$,
the Blaschke-Santal\'{o} inequality gives
\begin{equation*}\vol_k(P_F(B_1^n))^{\frac{1}{k}}\ls \vol_k(B_\infty^n\cap F)^{-\frac{1}{k}}\omega_k^{2/k}\ls \frac{c}{k}\end{equation*}
for all $k$ and $F\in G_{n,k}$. We summarize this preliminary information in the next lemma.

\begin{lemma}\label{lem:cross-1}For every $1\ls k\ls n$ and any $F\in G_{n,k}$ we have
\begin{equation*}\frac{c_1}{\sqrt{kn}}\ls \vol_k(P_F(B_1^n))^{\frac{1}{k}}\ls \frac{c_2}{k},\end{equation*}
where $c_1,c_2>0$ are absolute constants. In particular,
\begin{equation*}c_3\sqrt{n/k}\ls W_{[k,-p]}(B_1^n)\ls c_4\,n/k\end{equation*}
for every $p\neq 0$, where $c_3,c_4>0$ are absolute constants.
\end{lemma}

Next, we examine the typical behavior of a $k$-dimensional projection of $B_1^n$. For a random $F\in G_{n,k}$ we have the following upper bound:

\begin{lemma}\label{lem:cross-2}Let $1\ls k\ls n$. If $k\gr \log n$ then with probability greater than $1-\exp (-k)$ we have
\begin{equation*}\vol_k(P_F(B_1^n))^{1/k}\ls \frac{c\sqrt{\log (1+n/k)}}{\sqrt{kn}}.\end{equation*}
\end{lemma}

\begin{proof}We combine two well-known facts. The first one is the (upper estimate in the) Johnson-Lindenstrauss lemma from \cite{Johnson-Lindenstrauss-1982} .
\begin{quote}{\sl There exist absolute constants $c_i>0$, $i=1,2,3$, such that if $\varepsilon >0$ and $N<\exp (\varepsilon^2k/16)$ then
for every $\{ y_1,\ldots ,y_N\}\subset S^{n-1}$ there exists a set
${\cal G}\subseteq G_{n,k}$ of measure $\nu_{n,k}({\cal G})\gr 1-\exp (-\varepsilon^2k/16)$ such that for every $F\in {\cal G}$
and all $1\ls j\ls N$ we have}
\begin{equation}\label{eq:JL}\|P_F(y_j)\|_2\ls (1+\varepsilon )\sqrt{k/n}\end{equation}
\end{quote}
\noindent We also use well-known lower bounds for the volume of the intersection of a finite number of strips. Carl-Pajor \cite{Carl-Pajor-1988},
and independently Gluskin \cite{Gluskin-1988} (see also \cite{Barany-Furedi-1988}), obtained a lower bound for the volume of a symmetric polyhedron $K=\{ x\in {\mathbb R}^n:|\langle x,w_i\rangle
|\ls 1,i=1,\ldots ,N\}$ in terms of $\max\{ \|w_i\|_2:1\ls i\ls N\}$:
\begin{quote}
{\sl Let $w_1,\ldots ,w_n$ be vectors spanning ${\mathbb R}^k$ with $\|w_j\|_2\ls 1$ for all $1\ls j\ls n$. Consider the symmetric convex
body $C=\{ x\in {\mathbb R}^k:|\langle x,w_j\rangle |\ls 1,j=1,\ldots ,n\}$. Then,
\begin{equation}\label{eq:GL}\vol_k(C)^{1/k}\gr
\frac{c}{\sqrt{\log (1+n/k)}},\end{equation} where $c>0$ is an
absolute constant.}
\end{quote}
\noindent Consider the standard orthonormal vectors $e_1,\ldots ,e_n$ in ${\mathbb R}^n$. Let $\varepsilon >0$ to be chosen and consider
a subspace $F\in G_{n,k}$ that satisfies $\|P_F(e_j)\|_2\ls (1+\varepsilon )\sqrt{k/n}$ for all $1\ls j\ls n$. If we set $w_j=\frac{1}{1+\varepsilon }\sqrt{\frac{n}{k}}P_F(e_j)$ we have that
$$\max_{1\ls j\ls n}\|w_j\|_2\ls 1.$$
From \eqref{eq:GL} we get that $C=\{ x\in F:|\langle x,w_i\rangle |\ls
1,i=1,\ldots ,n\}$ has volume
\begin{equation*}\vol_k(C)^{1/k}\gr \frac{c_1}{\sqrt{\log (1+n/k)}}.\end{equation*}
Therefore, the polar body $C^{\circ }={\rm conv}(\{ \pm w_1,\ldots ,\pm w_n\})$ of $C$ in $F$ has volume
\begin{equation*}\vol_k(C^{\circ })^{1/k}\ls \frac{c_2}{k}\vol_k(C)^{-1/k}\ls \frac{c_3\sqrt{\log (1+n/k)}}{k}.\end{equation*}
Note that
$$P_F(B_1^n)={\rm conv}(\{\pm P_F(e_1),\ldots ,\pm P_F(e_n)\})=(1+\varepsilon)\sqrt{k/n}\,C^{\circ }.$$
It follows that
$$\vol_k(P_F(B_1^n))^{1/k}\ls  (1+\varepsilon)\frac{c_3\sqrt{\log (1+n/k)}}{\sqrt{kn}}$$
with probability greater than $1-\exp (-\varepsilon^2k/16)$, provided that $n<\exp (\varepsilon^2k/16)$. If $k\gr\log n$
then choosing $\varepsilon =4$ we get the lemma. \end{proof}

\smallskip

From Lemma \ref{lem:cross-2} we easily deduce a strengthened version of Theorem \ref{th:unco-1} for the cross-polytope, which in turn implies Theorem \ref{th:unco-1}.

\begin{theorem}\label{th:cross-1}Let $1\ls k\ls n$. If $k\gr \log n$ then
\begin{equation*}W_{[k,-p]}(B_1^n)\ls c\sqrt{n/k}\sqrt{\log (1+n/k)}\end{equation*}
for all $p\gr c^{\prime }/(n\log n)$, where $c,c^{\prime }>0$ are absolute constants. In particular,
\begin{equation*}\Phi_{[k]}(B_1^n)\ls c\sqrt{n/k}\sqrt{\log (1+n/k)}.\end{equation*}
\end{theorem}

\begin{proof}Let $A_{n,k}$ be the subset of $G_{n,k}$ on which we have
\begin{equation*}\vol_k(P_F(B_1^n))^{1/k}\ls \frac{c\sqrt{\log (1+n/k)}}{\sqrt{kn}}.\end{equation*}
From Lemma \ref{lem:cross-2} we have $\nu_{n,k}(A_{n,k})\gr 1-\exp (-k)$. Now given $p>0$ we may write
\begin{align*}W_{[k,-p]}(B_1^n) &\ls c_1n\left (\int_{G_{n,k}}\vol_k(P_F(B_1^n))^{-p}\,d\nu_{n,k}(F)\right )^{-\frac{1}{kp}}
\ls c_1n[\nu_{n,k}(A_{n,k})]^{-\frac{1}{kp}}\left (\frac{c\sqrt{\log (1+n/k)}}{\sqrt{kn}}\right )\\
&\ls c_2\sqrt{n/k}\sqrt{\log (1+n/k)}(1-e^{-k})^{-\frac{1}{kp}}\ls c_3\sqrt{n/k}\sqrt{\log (1+n/k)},
\end{align*}
if we take into account the fact that
\begin{equation*}(1-e^{-k})^{\frac{1}{kp}}\gr \exp (-2e^{-k}/(kp))\gr\frac{1}{2},\end{equation*}
because $ke^{k}\gr n\log n$ and $p\ls c/(n\log n)$.
\end{proof}

\smallskip

We pass to the proof of Theorem \ref{th:unco-2}. Our argument is based on the existence of $k$-dimensional subspaces of ${\mathbb R}^n$
for which $\vol_k(P_F(B_1^n))^{1/k}\ls \frac{c}{\sqrt{kn}}$. For the other extremum, it is not hard to give examples of $F\in G_{n,k}$ such that
$\vol_k(P_F(B_1^n))^{1/k}\asymp \frac{1}{k}$. We can simply choose $F_{\sigma }={\rm span}\{e_j:j\in\sigma \}$ for any $\sigma\subseteq [n]$ with $|\sigma |=k$. Then,
$$\vol_k(P_{F_{\sigma }}(B_1^n))=\frac{2^k}{k!}.$$
The next lemma provides concrete examples of subspaces $F\in G_{n,k}$ for which $\vol_k(P_F(B_1^n))^{1/k}\ls\frac{c}{\sqrt{kn}}$. We may assume that $k<\frac{n}{10}$,
otherwise this estimate holds for a random $F$ by Lemma \ref{lem:cross-2}.

\begin{lemma}\label{lem:cross-3}Let $1\ls k\ls n/10$. There exists $F\in G_{n,k}$ such that
$$\vol_k(P_F(B_1^n))^{1/k}\ls \frac{c}{\sqrt{kn}},$$
where $c>0$ is an absolute constant.
\end{lemma}

\begin{proof}We consider a partition $[n]=\sigma_1\cup\cdots\cup\sigma_k$ into disjoint
subsets $\sigma_i$ with $|\sigma_i|=m_i$, and we define
$$v_i=\frac{1}{\sqrt{m_i}}\sum_{j\in\sigma_i}e_j.$$
We may choose $m_1=\cdots =m_{k-1}=\lfloor n/k\rfloor $ and $m_k=n-(k-1)\lfloor n/k\rfloor $.
Note that $ \frac{n}{2k}\ls\frac{n}{k}-1\ls m_i\ls \frac{n}{k}$ for all $i=1,\ldots ,k-1$ and
$$m_k=n-(k-1)\lfloor n/k\rfloor\gr n-\frac{k-1}{k}n=\frac{n}{k}.$$
Let $F={\rm span}\{ v_1,\ldots ,v_k\}$. Observe that $v_1,\ldots ,v_k$ form an orthonormal basis for $F$ and that if $j\in\sigma_i$ then
$$P_F(e_j)=\langle e_j,v_i\rangle v_i=\frac{v_i}{\sqrt{m_i}}.$$
Therefore, $P_F(B_1^n)$ is the absolute convex hull of $k$ orthogonal vectors of lengths $\frac{1}{\sqrt{m_1}},\ldots ,\frac{1}{\sqrt{m_k}}\ls \sqrt{\frac{2k}{n}}$. It follows that
$$\vol_k(P_F(B_1^n))=\frac{2^k}{k!}\prod_{i=1}^k\frac{1}{\sqrt{m_i}}\ls \frac{2^k}{k!}\left (\frac{2k}{n}\right )^{k/2},$$
and hence $\vol_k(P_F(B_1^n))^{1/k}\ls c/\sqrt{kn}$. \end{proof}

\smallskip

The next lemma follows easily from the definition of the convex hull.

\begin{lemma}\label{lem:cross-4}Let $P={\rm conv}(\{ u_1,\ldots ,u_N\})$ and $Q={\rm conv}(\{ w_1,\ldots ,w_N\})$ be two polytopes in ${\mathbb R}^k$.
Assume that for some $\varepsilon >0$ we have
$\|u_j-w_j\|_2\ls\varepsilon $ for all $j=1,\ldots ,N$. Then,
$$P\subseteq Q+\varepsilon B_2^k\quad\hbox{and}\quad Q\subseteq P+\varepsilon B_2^k.$$
\end{lemma}

We consider the metrics $\sigma_{\infty }(E,F)=\|P_E-P_F\|$ and $d(E,F)=\inf\{\|I_n-U\|:U\in O(n),U(E)=F\}$ on $G_{n,k}$. We will use
the fact that
$$\sigma_{\infty }(E,F)\ls d(E,F)\ls\sqrt{2}\sigma_{\infty }(E,F)$$
for all $E,F\in G_{n,k}$. First, we fix a subspace $F_0$ that satisfies the estimate of Lemma \ref{lem:cross-3}.

\begin{lemma}\label{lem:cross-5}Let $E$ in $G_{n,k}$ with $d(E,F_0)\ls \frac{1}{\sqrt{n}}$. Then,
$$\vol_k(P_E(B_1^n))^{1/k}\ls \frac{c}{\sqrt{kn}},$$
where $c>0$ is an absolute constant.
\end{lemma}

\begin{proof}Let $U\in O(n)$ such that $U(E)=F_0$ and $\|I_n-U\|\ls\epsilon :=\frac{1}{\sqrt{n}}$. For every $j=1,\ldots ,n$ we set $u_j=P_{F_0}(e_j)$ and $w_j=U(P_E(e_j))$. Then,
$u_j,w_j\in F_0$ and we have
$$\|P_E(e_j)-w_j\|_2= \|P_E(e_j)-U(P_E(e_j))\|_2\ls d(E,F_0)\,\|P_E(e_j)\|_2\ls \epsilon $$
and
$$\|P_E(e_j)-u_j\|_2=\|P_E(e_j)-P_{F_0}(e_j)\|_2\ls \|P_E-P_{F_0}\|=\sigma_{\infty }(E,F_0)\ls \epsilon ,$$
which implies that
$$\|u_j-w_j\|_2\ls 2\epsilon =\frac{2}{\sqrt{n}}.$$
From Lemma \ref{lem:cross-4} we get
$$U(P_E(B_1^n))\subseteq P_{F_0}(B_1^n)+\frac{2}{\sqrt{n}}B_{F_0}\subseteq 3P_{F_0}(B_1^n).$$
Therefore,
$$\vol_k(P_E(B_1^n))^{1/k}\ls 3\vol_k(P_{F_0}(B_1^n))^{1/k}\ls \frac{3c}{\sqrt{kn}},$$
where $c>0$ is the constant in Lemma \ref{lem:cross-3}. \end{proof}

\begin{remark}\rm It was proved by Szarek in \cite{Szarek-1982} that for every $F\in G_{n,k}$ and any $\varepsilon >0$ one has
$$\nu_{n,k}(B_d(F,\varepsilon ))\gr (c_1\varepsilon )^{k(n-k)}.$$
Therefore, the upper bound
$$\vol_k(P_E(B_1^n))^{1/k}\ls \frac{C}{\sqrt{kn}}$$
holds true with probability greater than $(c_1/\sqrt{n})^{k(n-k)}$.
It follows that if $p>0$ then
\begin{align}W_{[k,-p]}(B_1^n) &\ls c_1n\left (\int_{B_d(F_0,1/\sqrt{n})}\vol_k(P_E(B_1^n))^{-p}d\nu_{n,k}(E)\right )^{-\frac{1}{kp}}\\
\nonumber &\ls c_1n\,[\nu_{n,k}(B_d(F_0,1/\sqrt{n}))]^{-\frac{1}{kp}}\cdot\frac{c_2}{\sqrt{kn}}\ls (c_3n)^{\frac{n-k}{p}}c_4\sqrt{n/k}.
\end{align}
In particular we get: For every $p\gr (n-k)(\log n)$ we have
$$W_{[k,-p]}(B_1^n)\ls c\sqrt{n/k}$$
where $c>0$ is an absolute constant, and hence $W_{[k,-p]}(B_1^n)\asymp \sqrt{n/k}$ by Lemma $\ref{lem:cross-1}$.
This proves Theorem \ref{th:unco-2}.
\end{remark}

\smallskip

\noindent {\bf Acknowledgements.} We would like to thank Apostolos Giannopoulos for useful discussions. The second named author is supported by a PhD Scholarship from the Hellenic Foundation for Research and Innovation (ELIDEK); research number 70/3/14547.

\footnotesize
\bibliographystyle{amsplain}

\medskip

\medskip

\medskip

\noindent{\bf Keywords:} convex bodies, affine quermassintegrals, random polytopes, asymptotic shape.
\\
\thanks{\noindent {\bf 2010 MSC:} Primary 52A23; Secondary 46B06, 52A40, 60D05.}

\medskip

\medskip

\noindent \textsc{Giorgos \ Chasapis}: Department of
Mathematical Sciences, Kent State University, Kent, OH 44242, USA.

\smallskip

\noindent \textit{E-mail:} \texttt{gchasap1@kent.edu}

\medskip

\noindent \textsc{Nikos \ Skarmogiannis}: Department of
Mathematics, University of Athens, Panepistimioupolis 157-84,
Athens, Greece.

\smallskip

\noindent \textit{E-mail:} \texttt{nikskar@math.uoa.gr}

\medskip

\end{document}